\documentclass[12pt]{amsart}
\usepackage{amssymb, amscd}

\usepackage{mathptmx}

\usepackage[latin1]{inputenc}
\newcommand{\R}{\mathbb{R}}
\newcommand{\N}{\mathbb{N}}

\newcommand{\Bc}{\mathcal{B}}

\newcommand{\Gc}{\mathcal{G}}

\newcommand{\Wc}{\mathcal{W}}

\DeclareMathOperator{\pnt}{\raise 0.5mm \hbox{\large\bf.}}

\DeclareMathOperator{\cone}{cone}

\DeclareMathOperator{\GL}{GL}
\DeclareMathOperator{\inom}{in}
\DeclareMathOperator{\gT}{gT}

\DeclareMathOperator{\gin}{gin}
\DeclareMathOperator{\rank}{rank}

\DeclareMathOperator{\wt}{wt}

\DeclareMathOperator{\GF}{GF}
\DeclareMathOperator{\gGF}{gGF}

\DeclareMathOperator{\charac}{char}

\newtheorem{theorem}{\bf Theorem} [section]
\newtheorem{lemma}[theorem]{\bf Lemma}
\newtheorem{cor}[theorem]{\bf Corollary}
\newtheorem{prop}[theorem]{\bf Proposition}

\theoremstyle{definition}
\newtheorem{defn}[theorem]{\bf Definition}

\newtheorem{rem}[theorem]{\bf Remark}

\newtheorem{notation}[theorem]{\bf Notation}
\theoremstyle{plain}
\newtheorem*{satz*}{Theorem}

\title{Generic Tropical Varieties}

\author{Tim R\"omer}
\address{Universit\"at Osnabr\"uck, Institut f\"ur Mathematik, 49069 Osnabr\"uck, Germany}
\email{troemer@uos.de}

\author{Kirsten Schmitz}
\address{Universit\"at Osnabr\"uck, Institut f\"ur Mathematik, 49069 Osnabr\"uck, Germany}
\email{kischmit@uos.de}

\begin{document}
\begin{abstract}
We show that in the constant coefficient case the generic tropical variety of a graded ideal exists. This can be seen as the analogon
to the existence of the generic initial ideal in Gr\"obner basis theory. We determine the generic tropical variety as a set in general and as a fan for principal ideals and linear ideals.
\end{abstract}


\maketitle
\section{Introduction}

The field of tropical geometry is a growing branch of mathematics
establishing a deep connection between algebraic geometry and
combinatorics. There are various different approaches and
applications of tropical geometry; see  \cite{DEST, GAMA, MI1, SPST} and for general overviews see \cite{GA, ITMISH}.

One important aspect of tropical geometry is that it provides a tool
to investigate affine algebraic varieties by studying certain
combinatorial objects associated to them. This is done by
considering the image of an affine algebraic variety $X$ under a
valuation map; see \cite{DR,JEMAMA,SPST}. The set of real-valued
points of this image is defined to be the tropical variety of $X$
or, equivalently, of the ideal $I$ defining $X$. The tropical
variety has the structure of a polyhedral complex in $\R^n$ and can
be used to obtain information of the original variety as is done for
example in \cite{DR}. For practical purposes there is a useful
characterization of tropical varieties in terms of initial
polynomials given in \cite{SPST} and fully proved in \cite[Theorem
4.2]{DR} and more explicitly in \cite{JEMAMA}. From this it follows
that in the case of constant coefficients, i.e.\ if the valuation on
the ground field is trivial, the tropical variety of an algebraic
variety defined by a graded ideal $I$ is a subfan of the Gr\"obner
fan of $I$. It contains exactly those cones of the Gr\"obner fan
corresponding to initial ideals that do not contain a monomial.

Let $K$ be an infinite field, $I\subset S_K=K[x_1,\ldots,x_n]$ a graded ideal and
$\succ$ a term order. It is well known that there exists a generic
initial ideal $\gin_{\succ}(I)$ with respect to $\succ$. More
precisely, there is a non-empty Zariski-open set $U\subset \GL_n(K)$
such that $\inom_{\succ}(g(I))$ is the same ideal for every $g\in
U$. This will be made precise in Definition \ref{gendef}; see also \cite{EI} or \cite{GR} for details and see for example
\cite{BAST, HE:survey} for applications of this concept in algebraic
geometry and commutative algebra. Since the tropical variety of $I$
is closely related to the Gr\"obner fan of $I$ and thus to initial
ideals of $I$, the question arises, whether there exists a generic
tropical variety of $I$ analogous to $\gin_{\succ}(I)$ and what
properties it has.

Our aim is to study the tropical variety of a graded ideal under a generic coordinate transformation. We prove the existence of a generic Gr\"obner fan and a generic tropical variety in the case of constant coefficients. Moreover, we explicitly describe the generic tropical variety of an ideal as a set. This set only depends on the dimension $m$ of the coordinate ring $S_K/I$. It is equal to the support of the $m$-skeleton $\Wc_n^m$ of one particular fan $\Wc_n$ in $\R^n$ (see Definition \ref{gentropfandefn}). The following main results of this paper are restated in Corollary \ref{gengrobi} and Theorem \ref{tropfansetcor}.

\begin{theorem}
Let $I\subset S_K=K[x_1,\ldots,x_n]$ be a graded ideal with $\dim(S_K/I)=m$. Then there exists a Zariski-open subset $\emptyset\neq U\subset \GL_(K)$, such that
\begin{enumerate}
\item the Gr\"obner fan $\GF(g(I))$ of the ideal $g(I)$ is the same fan for every $g\in U$,
\item the tropical variety $T(g(I))$ of $g(I)$ is the same fan for every $g\in U$ and this fan is supported by the underlying set of $\Wc_n^m$. In addition, every ideal has a generic tropical basis.
\end{enumerate}
\end{theorem}

The latter result yields a way to associate a non-empty tropical variety to an ideal of dimension at least one,
even if it contains a monomial. This opens the possibility to study such ideals by means of tropical varieties as well.  Note that the existence of a generic tropical variety highly depends on the fact that we use the constant coefficient case. The existence result is false in the general setting; see Remark \ref{exnonconstfalse}.

Our paper is organized as follows. In Section 2 we will introduce our notation and the basic setting for our work. In Section 3 we
present a proof of the existence of the generic Gr\"obner fan in this setting. Section 4 contains the proof of the main theorem regarding generic tropical varieties. In the last Section the example classes of principal ideals and linear ideals are discussed. We refer to \cite{ROSC} for further results on generic tropical varieties, like the relationship between the multiplicity of a generic tropical variety (see, e.g., \cite{DIFEST} or \cite{STTE} for the definition) and the multiplicity of the defining ideal.

We thank Hannah Markwig and Bernd Sturmfels for valuable comments and suggestions for this paper.

\section{Basic Concepts and Notation}\label{basicconcepts}

In this section we present some results and recall definitions which are used in the subsequent sections. Let $K$ be an infinite field. In general, for the purposes of tropical geometry $K$ is equipped with a non-archimedean valuation $v\colon K \to \R \cup \{\infty\}$, which induces the transition map between classical and tropical varieties. In this note we only consider the constant coefficient case, i.e.\ that $v(a)=0$ for all $a\in K^*$. This reduces the tropical geometry in our setting to the study of Gr\"obner fans (at least in characteristic 0); see Remark \ref{exnonconstfalse} for a hint at the general situation.
Note that the definition of a tropical variety as given below works in any characteristic and for the results of this paper only $|K|=\infty$ is required.

We will denote the polynomial ring in $n$ variables over $K$ by $S_K$. For a polynomial $f\in S_K$ with $f=\sum_{\nu\in\N^n} a_{\nu} x^{\nu}$ and $\omega\in\R^n$ we denote by $\inom_{\omega}(f)$ the
\emph{initial polynomial of $f$}, which consists of all terms of
$f$ such that $\omega\cdot\nu$ is minimal. Note that our definition is slightly different from the original one in the context of Gr\"obner basis theory, since for a given polynomial we always take terms of lowest $\omega$-weight, while one usually takes terms of maximal $\omega$-weight. However, this does not change the theory at all for the case of graded ideals. We use the above definition, since it is consistent with the definition of initial forms in the non-constant coefficient case. If the valuation on $K$ is non-trivial, the valuations of the coefficients $a_{\nu}$ are taken into account in the definition of $\inom_{\omega}(f)$, see \cite{DR} or \cite{SPST} for two such variations.

The \emph{tropical variety} $T(I)$ of a graded ideal $I\subset S_K$ is the set of all $\omega\in \R^n$ such that the minimal weight of the terms of $f$ is attained at least twice for all $f\in I$.
In other words, we have
$$
T(I)=\left\{\omega\in \R^n: \inom_{\omega}(f) \text{ is not
a monomial for every } f\in I\right\}.
$$
If $I=(f)$ is principal, we also write $T(f)$ for $T(I)$.

In the constant coefficient case the tropical variety of an ideal has a natural fan structure. Recall that a fan $\mathcal{F}$ in $\R^n$ is a finite collection of (polyhedral) cones in $\R^n$ such that for $C' \subset C$ with $C\in \mathcal{F}$ we have that $C'$ is a face of $C$ if and only if $C' \in \mathcal{F}$, and secondly if $C,C' \in \mathcal{F}$, then $C\cap C'$ is a common face of $C$ and $C'$. To simplify notation we denote by $\mathcal{F}$ also the union of all its cones. The dimension $\dim \mathcal{F}$ of $\mathcal{F}$ is the maximum of the dimensions $\dim C$ for all cones $C\in \mathcal{F}$ in the usual topology of $\R^n$. We call the fan \emph{pure-dimensional} if every maximal cone has the same dimension $\dim \mathcal{F}$.

In the following we will always assume $I$ to be a graded ideal with $I\neq\left\{ 0\right\}$, if not stated otherwise. Recall in this situation the notion of the \emph{Gr\"obner fan $\GF(I)$} of  $I$; see for example \cite{MATH}, \cite{MORO} or \cite{ST}. For $\omega\in \R^n$ we let $\inom_{\omega}(I)$ be the ideal generated by all $\inom_{\omega}(f)$ for $f \in I$. Two vectors $\omega,\omega'\in \R^n$ are elements of the same relatively open cone $\mathring{C}$ for $C \in \GF(I)$ if and only if $\inom_{\omega}(I)=\inom_{\omega'}(I)$. Then we set $\inom_{C}(I)$ for this common initial ideal.

It was observed in \cite{ST02} that the tropical variety $T(I)$ is a subfan of the Gr\"obner fan of $I$ in a natural way (see also \cite{BOJESPSTTH}). More precisely, we have:

\begin{prop}
\label{prop:gfan}
The tropical variety $T(I)$ of a graded ideal $I\subset S_K$ is the subfan of the Gr\"obner fan $\GF(I)$ which contains all cones $C \in\GF(I)$ such that the corresponding initial ideal $\inom_C (I)$ contains no monomial.
\end{prop}

The next basic result on tropical varieties is a direct consequence of the definition.

\begin{lemma}\label{tropbasics}
Let $I,J\subset S_K$ be graded ideals with $I\subset J$. If we consider the tropical varieties of $I$ and $J$ as sets, we have $T(J)\subset T(I)$. In particular, for a homogeneous polynomial $f\in I$ we have $T(I)\subset T(f)$.
\end{lemma}

To compute tropical varieties the concept of a tropical basis is useful. Let $I\subset S_K$ be a graded ideal. Then a finite
system of homogeneous generators $f_1,\ldots,f_t$ of $I$ is called a \emph{tropical basis} of $I$ if $$T(I)=\bigcap_{i=1}^t T(f_i).$$

Every ideal has a tropical basis. See, e.g., \cite[Theorem 2.9]{BOJESPSTTH} for the constant coefficient case and \cite{HETH} for the general case.

We will now specify the meaning of the term \emph{generic} for this note and introduce the notation used here.

\begin{defn}\label{gendef}
Let $G=\left\{y_{ij}: i,j=1,\ldots,n\right\}$ be a set of $n^2$
independent variables over some field $K$ and let $K'=K(G)$ be the
quotient field of $K[G]$. In the following we denote by $y$ the
$K$-algebra homomorphism
\begin{eqnarray*}
y: K[x_1,\ldots,x_n] & \longrightarrow & K'[x_1,\ldots, x_n]\\
                 x_i & \longmapsto     & \sum_{j=1}^n y_{ij} x_j.
\end{eqnarray*}
For any $g=(g_{ij})\in \GL_n(K)$ this induces a $K$-algebra automorphism on
$K[x_1,\ldots,x_n]$ by substituting $g_{ij}$ for $y_{ij}$. We identify $g$ with the induced
automorphism and use the notation $g$ for both of them.
\end{defn}

\begin{notation}\label{noty}
A polynomial $f\in K'[x_1,\ldots,x_n]$ will sometimes be denoted as $f(y)$ to emphasize its dependence on the variables $y_{ij}\in G$. Let $f(y)\in K'[x_1,\ldots,x_n]$ and $g\in \GL_n(K)$ such that no denominator in the coefficients of the monomials $x_1^{\nu_1}\ldots x_n^{\nu_n}$ vanishes when the $g_{ij}$ are substituted for the $y_{ij}$. Then we will denote the polynomial in $K[x_1,\ldots,x_n]$ obtained by this substitution by $f(g)$.
\end{notation}

The \emph{dimension} $\dim (S_K/I)$ for an ideal $I\subset S_K$ always refers to the Krull dimension of the coordinate ring $S_K/I$.  Note that for any $g\in\GL_n(K)$ the ideal $g(I)$ is a graded ideal of the same dimension as $I$. If $\dim(S_K/I)>0$, generically the tropical variety of $I$ is non-empty.

\begin{lemma}\label{dimo}
Let $I\subset S_K$ be a graded ideal with $\dim (S_K/I)>0$. Then there exists a Zariski-open set
$\emptyset\neq U\subset \GL_n(K)$ such that $T(g(I))\neq \emptyset$ for every $g\in U$.
\end{lemma}

\begin{proof}
We have to show that $g(I)$ contains no monomial for all $g$ in a
non-empty Zariski-open set $U\subset \GL_n(K)$. If $g(I)$
contains a monomial $x^\alpha$ for a fixed $g$, we would have
$(x^\alpha)\subset g(I)$, which implies the inclusions
$$V(g(I))\subset V(x^\alpha)=\left\{z\in K^n: z_i=0 \text{ for }\alpha_i>0\right\}$$
of the zero-sets of the two ideals. Thus it suffices to
show that there is a zero of $g(I)$, none of whose coordinates is
zero to show that no monomial can be contained in $g(I)$.

If $I=(f_1,\ldots,f_r)$, then $g(I)=(g(f_1),\ldots, g(f_r))$. Since
$g\in\GL_n(K)$, we can also consider it as a vector space
isomorphism on $K^n$. Let $g^{-1}$ denote its inverse. Then by
definition $g(f_i)(v)=f_i(g(v))$ for any $v\in K^n$. Thus for any
$z\in V(I)$ we get $$g(f_i)(g^{-1}(z))=f_i(g(g^{-1}(z)))=f_i(z)=0,$$
so $g^{-1}(z)\in V(g(I))$.

Since $\dim (S_K/I)>0$, we know $\sqrt{I}\neq(x_1,\ldots,x_n)$. In particular,
there exists $0\neq z\in V(I)$ because we are assuming that $K$ is algebraically closed. Now the $i$-th coordinate
$(g^{-1}(z))_i$ is zero if and only if $\sum_{j=1}^n g'_{ij}z_j=0$, where
the $g'_{ij}$ are the entries of the matrix of $g^{-1}\in \GL_n(K)$.
This sum can be considered as a non-zero polynomial in the variables
$g'_{ij}$ with coefficients $z_j$. Now we can choose $U$ to be
the set $$U=\left\{g\in \GL_n(K): \sum_{j=1}^n g'_{ij}z_j\neq0
\text{ for } i=1,\ldots,n\right\},$$ which is non-empty and
Zariski-open. Then for any $g\in U$ we have $g^{-1}(z)\in
V(g(I))\cap (K^*)^n$, so $g(I)$ cannot contain a monomial. Hence,
$T(g(I))\neq\emptyset$ for $g\in U$.
\end{proof}

Let $\succ$ be a term order on $S_K=K[x_1,\ldots,x_n]$ with $x_1\succ x_2\succ\ldots\succ x_n$. Then the initial ideal of some ideal $I\subset S_K$ with respect to $\succ$ is constant under a generic coordinate transformation of $I$. In other words there is a
Zariski-open set $\emptyset\neq U\subset \GL_n(K)$ such that
$\inom_{\succ}(g(I))$ is the same ideal for every $g\in U$,
and this ideal is denoted by $\gin_{\succ}(I)$.

Let $B_n(K)\subset \GL_n(K)$ denote the Borel subgroup of $\GL_n(K)$,
i.e.\ all upper triangular matrices in $\GL_n(K)$. Then
for every $g\in B_n(K)$ we have
$g^T(\gin_{\succ}(I))=\gin_{\succ}(I)$, where $g^T$ is the transposed matrix of $g$. This fact is expressed by calling
$\gin_{\succ}(I)$ \emph{Borel-fixed}. In the case that $\charac(K)=0$ this condition is equivalent to $\gin_{\succ}(I)$ being \emph{strongly stable}; see \cite[Theorem 15.23]{EI}. This means that for any index $i\in\left\{1,\ldots,n\right\}$ and any monomial $x^{\nu}\in\gin_{\succ}(I)$ which is divisible by $x_i$, also the monomial $(x_j/x_i)x^{\nu}$ is in $\gin_{\succ}(I)$. This condition will be used repeatedly in the following explaining our assumption $\charac(K)=0$.

As explained above the tropical variety of $I$ is a subfan of the Gr\"obner fan of $I$ and thus closely related to initial ideals of $I$. This leads to the question, whether there exists a generic tropical variety of $I$ analogous to $\gin_{\succ}(I)$ and what it looks like, if it does exist.

\begin{defn}\label{tropdef}
Let $I\subset S_K$ be a graded ideal. If $T(g(I))$ is the same fan for all $g$ in a Zariski-open subset $\emptyset\neq U\subset \GL_n(K)$, then this fan is called the \emph{generic tropical variety} of $I$ and is denoted by $\gT(I)$.
\end{defn}

Note that every graded ideal $I\subset S_K$ with $\dim (S_K/I)=0$ contains a monomial. Thus Lemma \ref{dimo} immediately implies that we have $\gT(I)=\emptyset$ if and only if $\dim (S_K/I)=0$.

The \emph{support} of a polynomial $f$ is the finite set of all exponent vectors of $f$. More generally, the \emph{support} of a finite set $\Gc$ of polynomials is the union of the support-sets of every polynomial in $\Gc$. We would like to obtain tropical bases of $g(I)$ with the same support for all $g$ in some non-empty open subset of $\GL_n(K)$. This idea is captured in the next definition.

\begin{defn}\label{gentropbasis}
Let $I\subset S_K=K[x_1,\ldots,x_n]$ be a graded ideal. A finite set $\left\{f_1(y),\ldots,f_s(y)\right\}$ of polynomials in $y(I)$ is called a \emph{generic tropical basis of $I$}, if there is an open subset $\emptyset\neq U\subset \GL_n(K)$ such that $\left\{f_1(g),\ldots,f_s(g)\right\}$ is a tropical basis of $g(I)$ with the same support for every $g\in U$. If an open set $\emptyset\neq U\subset \GL_n(K)$ fulfills this condition, the generic tropical basis is said to be \emph{valid} on $U$.
\end{defn}

In Section \ref{gentropvarsec} it will be proved that generic tropical varieties exist and that every graded ideal has a generic tropical basis in the constant coefficient case.

\begin{rem}\label{exnonconstfalse}
Definition \ref{tropdef} can be formulated in the same way in the non-constant coefficient case, i.e. if the valuation $v$ on $K$ is non-trivial. In this case the initial form of a homogeneous polynomial $f\in K[x_1,\ldots,x_n]$ is defined by taking the valuations of the coefficients of $f$ into account; see e.g. \cite{SPST}. For example, for the linear form $f=g_{11}x+g_{12}y\in K[x,y]$, the initial form $\inom_{\omega}(f)$ is not a monomial, if and only if $v(g_{11})+\omega_1=v(g_{12})+\omega_2$.

This example suffices to show that the condition of Definition \ref{tropdef} will not be fulfilled in general in the constant coefficient case. We consider the ideal $I=(x)\subset K[x,y]$. Then $g(I)=(g_{11}x+g_{12}y)$, so if $g_{11},g_{12}\neq 0$, we get $$T(g(I))=\left\{\omega\in \R^2:v(g_{11})+\omega_1=v(g_{12})+\omega_2\right\}.$$ This affine subspace of $\R^2$ of course depends on the value of $v(g_{11})-v(g_{12})=v(\frac{g_{11}}{g_{12}})$ which will not the same for general $g_{11},g_{12}\in K$. Hence, there is no Zariski-open subset $U\subset \GL_2(K)$ such that $T(g(I))$ is the same set for every $g\in U$.
\end{rem}

\section{The Generic Gr\"obner Fan}\label{secgrofan}

In this section we show the existence of a ``generic Gr\"obner fan'' of a graded ideal $I \subset
S_K=K[x_1,\ldots,x_n]$.

Recall that $I$ has only finitely many initial ideals with respect
to term orders on the polynomial ring $K[x_1,\dots,x_n]$ and these
initial ideals correspond to the maximal cones in the Gr\"obner fan
of $I$. A \emph{universal Gr\"obner basis} of $I$ is a finite
generating set of $I$ which is a Gr\"obner basis of $I$ with respect
to every term order. Note that such a universal Gr\"obner basis
always exists. Indeed, choosing term orders $\succ_1,\dots,\succ_m$
such that $\inom_{\succ_1}(I),\dots,\inom_{\succ_m}(I)$ are all
initial ideals of $I$, then the union of all reduced Gr\"obner bases
of $I$ with respect to $\succ_i$ for $i=1,\dots,m$ is a universal
Gr\"obner basis of $I$; see for example \cite[Corollary
2.2.5]{MATH}.

Recall that $K'=K(y_{ij}:i,j=1,\ldots,n)$ as defined in Section \ref{basicconcepts}. We may identify term orders on $S_K$ with those on $S_{K'}=K'[x_1,\ldots,x_n]$. Moreover, we also
identify  monomial ideals in $S_K$ with those in
$K'[x_1,\ldots,x_n]$, since the monomials do not depend on the
ground field.

\begin{theorem}\label{groebfan}
Let $I\subset S_K$ be a graded ideal. There exists a Zariski-open subset $\emptyset\neq U\subset \GL_n(K)$ and polynomials $h_1(y),\ldots,h_s(y)\in  y(I)$ such that
\begin{enumerate}
\item $\Gc(y)=\left\{h_1(y),\ldots,h_s(y)\right\}$ is a universal Gr\"obner basis of $y(I)$.
\item For every $g\in U$ the set $\Gc(g)=\left\{h_1(g),\ldots,h_s(g)\right\}$ is a universal Gr\"ob\-ner basis of $g(I)$.
\item All these universal Gr\"ob\-ner bases have the same support.
\end{enumerate}
\end{theorem}

\begin{proof}
Let $J\subset K'[x_1,\ldots,x_n]$ be the image ideal $y(I)$ of $I$
under the $K$-algebra homomorphism $y$ as defined in Definition
\ref{gendef}. There exists only finitely many initial ideals
$\inom_1(J),\ldots,\inom_m(J)$ of $J$ with respect to term orders of
$K'[x_1,\dots,x_n]$. We choose a term order $\succ_i$ for each
initial ideal  $\inom_i(J)$ such that
$\inom_{\succ_i}(J)=\inom_i(J)$. Using the Buchberger Algorithm  we
can compute a reduced Gr\"obner basis $\Gc_i$ of $J$ with respect to
$\succ_i$. Let $\Gc(y)$ be the union of all these reduced Gr\"obner
bases $\Gc_i$ of $J$, i.e.\ a universal Gr\"obner basis of $J$. The
coefficients of all polynomials occurring throughout these
computations are themselves quotients of polynomials in the
variables $y_{ij}$. Now choose $U$ to be the non-empty Zariski-open
set of all $g\in\GL_n(K)$ such that all of the finitely many
numerators and denominators of the polynomials appearing during the
calculations in the algorithm are nonzero with respect to
any of the $\succ_i$. Then for any $g\in U$ the reduced Gr\"obner
basis $\Gc_i(g)$ of $g(I)$ with respect to $\succ_i$ is obtained by
evaluating the polynomials of $\Gc_i$ at $g$.

Now it remains to show that for $g\in U$ the union of the
$\Gc_i(g)$ is a universal Gr\"obner basis of $g(I)$. For this it is
enough to prove that every initial ideal of $g(I)$ is one of the
$\inom_1(J),\ldots,\inom_m(J)$.  Let $g\in U$ be fixed and $\succ$
be any term order and consider the initial ideal
$\inom_{\succ}(g(I))$. We know that $\inom_{\succ}(J)=\inom_{i}(J)$
for some $i\in\left\{1,\ldots,m\right\}$. This implies that the
reduced Gr\"obner basis $\Gc_i$ of $J$ with respect to $\succ_i$ is
also a reduced Gr\"obner basis of $J$ with respect to $\succ$; see
\cite[Corollary 2.2.5]{MATH}. Moreover, by the choice of $U$ we
know that $\Gc_i(g)$ is a reduced Gr\"obner basis of $g(I)$ with
respect to $\succ_i$ for $g\in U$. Since $\Gc_i$ and $\Gc_i(g)$ have
the same support, we know $\inom_{\succ}(y(f))=\inom_{\succ}(g(f))$
and $\inom_{\succ_i}(y(f))=\inom_{\succ_i}(g(f))$ for every $y(f)\in
\Gc_i$. We also know that
$\inom_{\succ}(y(f))=\inom_{\succ_i}(y(f))$, since
$\inom_{\succ}(J)=\inom_{\succ_i}(J)$ and $\Gc_i$ is reduced. But
then we get
\begin{eqnarray*}
\inom_{\succ_i}(g(I)) & = & (\inom_{\succ_i}(g(f)): g(f)\in\Gc_i(g))\\
                          & = & (\inom_{\succ_i}(y(f)): y(f)\in\Gc_i)\\
                          & = & (\inom_{\succ}(y(f)): y(f)\in\Gc_i)\\
                          & = & (\inom_{\succ}(g(f)): g(f)\in\Gc_i(g))
                          \subset  \inom_{\succ}(g(I)).
\end{eqnarray*}
However, both $\inom_{\succ_i}(g(I))$ and $\inom_{\succ}(g(I))$ are
initial ideals of the same ideal $g(I)$, and hence,
$\inom_{\succ}(g(I))=\inom_{\succ_i}(g(I))$.

This means that $\Gc(g)$ defined as the union of the $\Gc_i(g)$ for
$i=1,\ldots,m$ is a universal Gr\"obner basis of $g(I)$. Now
$\Gc(g)$ is obtained by evaluating the coefficients of the
polynomials in $\Gc$, and for $g\in U$ none of these coefficients
vanishes. Hence, all $\Gc(g)$ consist of polynomials which differ
only in the coefficients not equal to zero. So all $\Gc(g)$ for $g
\in U$ have the same support.
\end{proof}

Note that in particular this implies the well-known result that for a graded ideal $I\subset S_K$ there exist only finitely many generic initial ideals of $I$. As the Gr\"obner fan of $g(I)$ depends only on the support of the polynomials in the universal Gr\"obner basis, this also immediately implies the existence of a generic Gr\"obner fan.

\begin{cor}\label{gengrobi}
Every ideal $g(I)$ has the same Gr\"obner fan for every $g\in U$ for some non-empty open subset $U\subset \GL_n(K)$.
\end{cor}

Since every non-empty Zariski-open subset is dense in $\GL_n(K)$, the following definition makes sense.

\begin{defn}
The unique polyhedral fan that equals $\GF(g(I))$ for all $g$ in a non-empty Zariski-open subset of $\GL_n(K)$, is called the
\emph{generic Gr\"obner fan of $I$}. We denote this fan by $\gGF(I)$.
\end{defn}

We also state two Corollaries of Theorem \ref{groebfan} needed in Section \ref{gentropvarsec}.

\begin{cor}\label{ginlink}
Let $I\subset S_K$ be a graded ideal and $\succ$ a term order. Then $\inom_{\succ}(y(I))\subset S_{K'}$ and $\gin_{\succ}(I)\subset S_K$ have the same sets of minimal generators.
\end{cor}

\begin{proof}
The reduced Gr\"obner bases of $y(I)$ and $g(I)$ with respect to $\succ$ have the same support for every $g$ in a non-empty open subset of $\GL_n(K)$ by Theorem \ref{groebfan}.
\end{proof}

\begin{cor}\label{ginoffen}
Let $I\subset S_K$ be a graded ideal. Then there exists an open set $\emptyset\neq U\subset \GL_n(K)$ such that for every $\omega\in \R^n$, every term order $\succ$ and every $g\in U$ we have $\inom_{\succ}(\inom_{\omega}(g(I)))=\gin_{\succ_{\omega}}(I)$
\end{cor}

\begin{proof}
We claim that the set $U\subset \GL_n(K)$ from Theorem \ref{groebfan} has this property. Let $\omega\in\R^n$ and $\succ$ any term order. Let $\Gc(g)=\left\{h_1(g),\ldots,h_s(g)\right\}$ be the universal Gr\"obner basis of $g(I)$ with the same support for $g\in U$ existing by Theorem \ref{groebfan}. In particular, $\Gc(g)$ is a Gr\"obner basis of $g(I)$ with respect to $\succ_{\omega}$. Thus $\left\{\inom_{\omega}(h_1(g)),\ldots,\inom_{\omega}(h_s(g))\right\}$ is a Gr\"obner basis of $\inom_{\omega}(g(I))$  with respect to $\succ$. With Theorem \ref{groebfan} this implies
\begin{eqnarray*}
\inom_{\succ}(\inom_{\omega}(g(I))) & = & (\inom_{\succ}(\inom_{\omega}(h_1(g))),\ldots,\inom_{\succ}(\inom_{\omega}(h_s(g))))\\
                                    & = & (\inom_{\succ_{\omega}}(h_1(g)),\ldots,\inom_{\succ_{\omega}}(h_s(g)))\\
                                    & = & \inom_{\succ_{\omega}}(g(I))\\
                                    & = & \gin_{\succ_{\omega}}(I).
\end{eqnarray*}
\end{proof}

The generic Gr\"obner fan is symmetric with respect to coordinates in the following sense. Let $S_n$ denote the symmetric group of degree $n$. For $\sigma\in S_n$ and $\omega=(\omega_1,\ldots,\omega_n)\in \R^n$ we set $\sigma(\omega)=(\omega_{\sigma(1)},\ldots,\omega_{\sigma(n)})$. Moreover, $\sigma$ induces a $K$-algebra automorphism on $K[x_1,\ldots,x_n]$ by setting $\sigma(x_i)=x_{\sigma(i)}$. By abuse of notation this map will also be denoted by $\sigma$. For $g=(g_{ij})\in \GL_n(K)$ let $\sigma(g)=(g_{i\sigma^{-1}(j)})$. Hence, $\sigma(g)$ corresponds to a switching of the columns of the matrix of $g$. Note that with this notation for a graded ideal $I\subset K[x_1,\ldots,x_n]$ and $\sigma,\tau\in S_n$ we have

\begin{enumerate}
\item $\sigma(g(I))=\sigma(g)(I)$,
\item $\tau(\sigma(g))=(\sigma\circ\tau)(g))$.
\end{enumerate}

Furthermore, every non-empty Zariski-open subset of $\GL_n(K)$ contains an open subset which is symmetric with respect to renaming coordinates. This means that for an open set $\emptyset\neq U\subset \GL_n(K)$ we can choose an open set $\emptyset \neq V\subset U$ such that for every $\sigma\in S_n$ we have: $$g\in V\text{ implies } \sigma(g)\in V.$$

With this we can state a result on the symmetry of generic Gr\"obner fans.

\begin{prop}\label{symsym}
Let $I\subset K[x_1,\ldots,x_n]$ be a graded ideal and $\mathring{C}$ be a relatively open cone in $\gGF(I)$. Then
$$\sigma(\mathring{C})=\left\{\sigma(\omega): \omega\in \mathring{C} \right\}$$ is also a relatively open cone of $\gGF(I)$ for $\sigma \in S_n.$
\end{prop}

\begin{proof}
Let $\emptyset\neq V\subset \GL_n(K)$ be Zariski-open such that $\GF(g(I))=\gGF(I)$ for all $g\in V$ and such that $g\in V$ implies $\sigma(g)\in V$. Let $J$ be the initial ideal corresponding to $\mathring{C}$. Now we have $\omega\in \mathring{C}$ if and only if $\inom_{\omega}(g(I))=J$ for $g\in V$. As $\inom_{\sigma(\omega)}(\sigma(g(I)))$ is obtained from $\inom_{\omega}(g(I))$ by renaming coordinates, $\omega\in \mathring{C}$ is equivalent to $\inom_{\sigma(\omega)}(\sigma(g)(I))=\inom_{\sigma(\omega)}(\sigma(g(I)))=\sigma(J)$. Since $\sigma(g)\in V$, the ideal $\sigma(J)$ then also defines a cone of $\gGF(I)$. This cone contains exactly all $\sigma(\omega)$ for $\omega\in \mathring{C}$ in its relative interior.
\end{proof}

\section{Generic Tropical Varieties}\label{gentropvarsec}

The generic tropical variety of an ideal turns out to be closely connected to one particular fan in $\R^n$ which we describe first. Let $e_i$ denote the $i$th standard basis vector of $\R^n$ and $\cone(M)$ denote the positive hull of a set $M$.

\begin{defn}\label{gentropfandefn}
Let $\Wc_n$ be the fan in $\R^n$ consisting of the following closed cones: For each non-empty subset $A\subset \left\{1,\ldots,n\right\}$ let
$$C_A = \cone(\left\{e_i:i\notin A\right\})+\R (1,\ldots,1).$$ This fan will be called the \emph{generic tropical fan} in $\R^n$. The $t$-skeleton of $\Wc_n$ will be denoted by $\Wc_n^t$.
\end{defn}

Equivalently we can write $C_A=\left\{\omega\in \R^n: \omega_i=\min_k\left\{\omega_k\right\}\text{ for all } i\in A\right\}$. Note that the image of $\Wc_n$ in $\R^n/(1,\ldots,1)$ is a fan of the projective $(n-1)$-space as a toric variety.

For a $k$-dimensional cone $C_A$ of $\Wc_n$ the set $A$ has to have exactly $n-k+1$ elements. Thus the number of cones of dimension $k$ is equal to the number of possibilities to choose $n-k+1$ from $n$, which is $\binom{n}{n-k+1}=\binom{n}{k-1}$. Therefore, $\Wc_n$ has exactly $\binom{n}{k-1}$ cones of dimension $k$ for $k=1,\ldots,n$.

We now show that for an ideal $I\subset S_K=K[x_1,\ldots,x_n]$ with $\dim (S_K/I)=m$ generically the tropical variety is contained in the $m$-skeleton of the generic tropical fan. Recall the definition of the field $K'$ and the ideal $y(I)$ in $S_{K'}=K'[x_1,\ldots,x_n]$ from Definition \ref{gendef}.

\begin{lemma}\label{skeletor}
Let $I\subset S_K=K[x_1,\ldots,x_n]$ be a graded ideal with $\dim(S_K/I)=m<n$. Then there exist polynomials $f_1(y),\ldots,f_s(y)\in y(I)$, such that $\bigcap_{i=1}^s T(f_i(y))\subset \Wc_n^m$. In particular, $T(y(I))\subset \Wc_n^m$.
\end{lemma}

\begin{proof}
Since $y: S_K\rightarrow S_{K'}$ is a flat extension, we have $\dim(S_{K'}/y(I))=\dim(S_K/I)=m$.

In the case $m=0$ both $T(y(I))$ and $\Wc_n^m$ are empty, so let $m>0$. Let $\mathring{C}\in \GF(y(I))$ be a relatively open Gr\"obner cone of $y(I)$ such that $\mathring{C}\not\subset \Wc_n^m$. Choose $\omega\in \mathring{C}\backslash \Wc_n^m$, so the minimum of the coordinates of $\omega$ is attained at most $n-m$ times. Without loss of generality we may assume that $\min_i\left\{\omega_i\right\}=0$ and the first
$r$ coordinates $r\leq n-m$ attain the minimum.

Let $\succ$ be the lexicographic term order induced by $x_1\succ x_2\succ\ldots\succ x_n$ and let $\succ_{\omega}$ be the refinement of the partial order corresponding to $\omega$ with respect to $\succ$. Then $\gin_{\succ_{\omega}}(I)$ exists and we have $\dim (S_K/\gin_{\succ_{\omega}}(I))=\dim (S_K/I)=m$. In particular, $$\gin_{\succ_{\omega}}(I)\cap K[x_r,\ldots,x_n]\neq \left\{0\right\},$$ since otherwise $K[x_r,\ldots,x_n]$ would be subset of a Noether normalization of the ring $K[x_1,\ldots,x_n]/\gin_{\succ_{\omega}}(I)$ and therefore $\dim(S_K/I)\geq n-r+1\geq m+1$ which is a contradiction to the assumption $\dim(S_K/I)=m$.

Let $0\neq u\in \gin_{\succ_{\omega}}(I)\cap K[x_r,\ldots,x_n]$ be a monomial of total degree $t$. Since $\gin_{\succ_{\omega}}(I)$ is Borel-fixed, this implies $x_r^t\in \gin_{\succ_{\omega}}(I)$; see, e.g., \cite[Theorem 15.23]{EI}. Since $\gin_{\succ_{\omega}}(I)$ and $\inom_{\succ_{\omega}}(y(I))$ have the same minimal generators by Corollary \ref{ginlink}, we also have $x_r^t\in \inom_{\succ_{\omega}}(y(I))$. Let $f(y)\in y(I)$ such that $\inom_{\succ_{\omega}}(f(y))=x_r^t$. No term of $f(y)$ that
has the same $\omega$-weight as $x_r^t$ may contain a variable from
$x_1,\ldots,x_{r-1}$, since then $\inom_{\succ_{\omega}}(f(y))\neq
x_r^t$ in the chosen lexicographic term order. So every such term of
$f(y)$ apart from $x_r^t$ must be divisible by one of the variables
$x_{r+1},\ldots,x_n$. But then every term of $f(y)$ has
$\omega$-weight greater than zero, except $\wt_{\omega} (x_r^t)=0$.
Hence, $\inom_{\omega}(f(y))=x_r^t$ is a monomial. This implies $\omega\notin T(f(y))$. Thus $T(f(y))\subset \R^n\backslash \mathring{C}\cup \Wc_n^m$. Repeating this procedure for every Gr\"obner cone $C$ of $y(I)$ with $\mathring{C}\not\subset \Wc_n^m$ yields finitely many polynomials $f_1(y),\ldots,f_s(y)\in y(I)$ such that $\bigcap_{i=1}^s T(f_i(y))\subset \Wc_n^m$. By Lemma \ref{tropbasics} this implies $T(y(I))\subset \Wc_n^m$.
\end{proof}

\begin{cor}\label{skeletorg}
Let $I\subset S_K=K[x_1,\ldots,x_n]$ be a graded ideal with $\dim(S_K/I)=m<n$. Then there exists a non-empty open subset $U\subset \GL_n(K)$ such that for every $g\in U$ there is a set of polynomials $\left\{f_1(g),\ldots,f_s(g)\right\}\subset g(I)$ having the same support for every $g\in U$ with $\bigcap_{i=1}^s T(f_i(g))\subset \Wc_n^m$.
\end{cor}

\begin{proof}
Let $f_1(y),\ldots,f_s(y)\in y(I)$ be as in Lemma \ref{skeletor}. Choose $\emptyset\neq U\subset \GL_n(K)$ such that no numerator or denominator of the coefficients of the $f_i(y)$ vanishes, when the $g_{ij}$ are substituted for the $y_{ij}$. Then $\left\{f_1(g),\ldots,f_s(g)\right\}$ has the same support for $g\in U$. Moreover, $\bigcap_{i=1}^s T(f_i(g))\subset \Wc_n^m$ by Lemma \ref{skeletor} as a tropical hypersurface depends only on the support of its generator in the constant coefficient case.
\end{proof}

The next result is a converse to Corollary \ref{skeletorg}.

\begin{lemma}\label{diego}
Let $I\subset S_K$ be a graded ideal with $\dim(S_K/I)=m$. Then there exists an open subset $\emptyset\neq U\subset \GL_n({K})$ such that $\Wc_n^m\subset T(g(I))$ for every $g\in U$.
\end{lemma}

\begin{proof}
Let $\emptyset\neq U\subset \GL_n({K})$ be open, such that $\inom_{\succ}(\inom_{\omega}(g(I)))=\gin_{\succ_{\omega}}(I)$ for $g\in U$ for any $\omega\in \R^n$ and any term order $\succ$. Such a set exists by Corollary \ref{ginoffen}. We will show that the claim of the lemma holds for every $g\in U$.

Let $\omega\in \Wc_n^m$. For a fixed $g\in U$ let $P$ be a minimal prime of $\inom_{\omega}(g(I))$ with $\dim(S_K/P)=m$. Assume that $P$ contains a monomial. Since $P$ is prime, this implies that $P$ contains a variable $x_l$ for some $l$. Without loss of generality let $\omega_1=\ldots=\omega_{n-m+1}\leq\omega_j$ for $j>n-m+1$. To establish a contradiction let $\left\{i_1,\ldots,i_{n-m}\right\}\subset\left\{1,\ldots,n-m+1\right\}\backslash \left\{l\right\}$. Let $\succ$ be a lexicographic term order with $$x_{i_1}\succ x_{i_2}\succ \ldots \succ x_{i_{n-m}}\succ x_j \text{ for } j\notin \left\{i_1,\ldots,i_{n-m}\right\}.$$

By assumption we have $\gin_{\succ_{\omega}}(I)=\inom_{\succ}(\inom_{\omega}(g(I)))\subset \inom_{\succ}(P)$ with $$\dim(S_K/\gin_{\succ_{\omega}}(I))=\dim(S_K/\inom_{\succ}(P))=m.$$ Let $Q$ be a minimal prime of $\inom_{\succ}(P)$. Since the dimensions coincide, $Q$ is also a minimal prime of $\gin_{\succ_{\omega}}(I)$. But $\gin_{\succ_{\omega}}(I)$ has only one minimal prime which is $(x_{i_1},\ldots,x_{i_{n-m}})$ by the choice of the term order $\succ$ (see for example \cite[Corollary 15.25]{EI}). Hence, $Q$ does not contain $x_l$. This is a contradiction to the fact that $x_l\in P$ and therefore $x_l\in \inom_{\succ}(P)\subset Q$. Thus, $P$ cannot contain a monomial. Hence, $\inom_{\omega}(g(I))\subset P$ cannot contain a monomial implying $\omega\in T(g(I))$. Since this holds for every $g\in U$, this proves the claim.
\end{proof}

This implies the following characterization of generic tropical varieties as a set in the constant coefficient case.

\begin{theorem}\label{tropfansetcor}
Let $I\subset S_K=K[x_1,\ldots,x_n]$ be a graded ideal with $\dim(S_K/I)=m<n$. Then $\gT(I)$ exists and as a set $$\gT(I)=\Wc_n^m.$$ Moreover, there exists a generic tropical basis for $I$ (as in Definition \ref{gentropbasis}).
\end{theorem}

\begin{proof}
Let $\left\{f_1(g),\ldots,f_s(g)\right\}\subset g(I)$ be a finite set of polynomials having the same support for every $g$ in a non-empty open subset $U_1\subset \GL_n(K)$ such that $\bigcap_{i=1}^s T(f_i(g))\subset \Wc_n^m$ for every $g\in U_1$. This exists by Corollary \ref{skeletorg}. Moreover, let $\emptyset\neq U_2\subset \GL_n(K)$ be open such that $\Wc_n^m\subset T(g(I))$ for $g\in U_2$ existing by Lemma \ref{diego}. Then for $g\in U_1\cap U_2$ we have $$\Wc_n^m\subset T(g(I))\subset \bigcap_{i=1}^s T(f_i(g))\subset \Wc_n^m$$ implying $T(g(I))=\Wc_n^m$ for $g\in U_1\cap U_2$. Since $U_1\cap U_2$ is open, the generic tropical variety $\gT(I)$ exists and as a set is equal to $\Wc_n^m$.

In addition, let $\left\{h_1,\ldots,h_r\right\}$ be a set of generators of $I$. Let $U_3\subset \GL_n(K)$ be a non-empty open set such that the sets $\left\{g(h_1),\ldots,g(h_r)\right\}$ have the same support for every $g\in U_3$. Since $g(h_1),\ldots,g(h_r)$ generate $g(I)$ for every $g\in \GL_n(K)$ and by the equality $T(g(I))=\bigcap_{i=1}^s T(f_i(g))$ for $g\in U_1\cap U_2$, the set $$\left\{y(h_1),\ldots,y(h_r),f_1(y),\ldots,f_s(y)\right\}$$ is a tropical basis of $I$ valid on $U_1\cap U_2\cap U_3$.
\end{proof}

In particular, in the constant coefficient case the generic tropical variety of an ideal as a set depends only on its dimension. Moreover, as a Corollary we recover the statement of Bieri and Groves \cite{BIGR} that the Krull dimension of $S_K/I$ coincides with the topological dimension of $T(I)$ in the constant coefficient case in the generic situation.

\begin{cor}[Bieri and Groves]
Let $I\subset S_K$ be a graded ideal. Then there exists an open subset $\emptyset\neq U\subset \GL_n(K)$ such that $\dim(S_K/g(I))=\dim T(g(I))$ for every $g\in U$.
\end{cor}

\section{Examples}

We conclude this note with some examples of generic Gr\"obner fans and generic tropical varieties. We briefly discuss principal ideals and linear ideals.

To describe the generic tropical variety of principal ideals we first prove a simple auxiliary statement.

\begin{lemma}\label{puretermpoly}
For a given homogeneous polynomial $0\neq f\in S_K$ of
total degree $d$ we can find a non-empty Zariski-open set $U\subset
\GL_n(K)$ such that $g(f)$ contains all terms $P_k(g) x_k^d$ with
nonzero coefficients $P_k(g)$ for all $g\in U$.
\end{lemma}

\begin{proof}
Let $f=\sum_{\nu\in \N^n} a_{\nu} x_1^{\nu_1}\cdots x_n^{\nu_n}$
with $\sum_{i=1}^n \nu_i=d$. Then $$g(f)=\sum_{\nu\in \N^n} a_{\nu}
(\sum_{j=1}^n g_{1j} x_j)^{\nu_1}\cdots (\sum_{j=1}^n g_{nj}
x_j)^{\nu_n}.$$ So $g(f)$ contains the terms $(\sum_{\nu} a_{\nu}
g_{1k}^{\nu_1}\cdots g_{nk}^{\nu_n})  x_k^d$. Let $P_k(g)=\sum_{\nu}
a_{\nu} g_{1k}^{\nu_1}\cdots g_{nk}^{\nu_n}$. Because $f$ is not the
zero polynomial we can choose $U$ to be the set of all $g\in
\GL_n(K)$ with $P_k(g)\neq0$ for $k=1,\ldots,n$.
\end{proof}

\begin{prop}\label{prince}
Let $0\neq f \in S_K$ be a homogeneous polynomial.
Then:
\begin{enumerate}
\item $\gGF(f)$ is equal to the generic tropical fan $\Wc_n$.
\item $\gT(f)$ is equal to $\Wc_n^{n-1}$, the $(n-1)$-skeleton of the generic tropical fan.
\end{enumerate}
\end{prop}

\begin{proof}
We consider the Zariski-open set $\emptyset\neq U\subset \GL_n(K)$
such that $g(f)$ has the maximal number of terms for all $g\in U$,
i.e.\ $g$ is not a zero of any nonzero coefficient polynomial of the
terms in $g(f)$. In particular, by Lemma \ref{puretermpoly} we know
$P_k(g)\neq0$ for $k=1,\ldots,n$ for all $g\in U$. Since $g(f)$ is
homogeneous, this implies that $\inom_{\omega}(g(f))$ is exactly the
sum of those terms of $g(f)$, that contain only variables $x_i$ for
which $\omega_i=\min\left\{\omega_j: j=1,\ldots,n\right\}$. So for
$\omega,\omega'\in \R^n$ we have
$\inom_{\omega}(g(f))=\inom_{\omega'}(g(f))$ if and only if
$$\left\{i: \omega_i=\min\left\{\omega_j:
j=1,\ldots,n\right\}\right\}=\left\{i:
\omega'_i=\min\left\{\omega'_j: j=1,\ldots,n\right\}\right\}.$$
Hence, $\omega$ and $\omega'$ are in the same Gr\"obner cone of $g(I)$
if and only if they are in the same cone $\Wc_n$ for all $g\in U$
and we conclude $\gGF(f)=\Wc_n$.

For the computation of the generic tropical variety we note that
$\inom_{\omega}(g(f))$ is a monomial $P_k(g) x_k^d$ for $g\in U$, if
$\omega_k<\omega_j$ for all $j\neq k$. If the minimum on the other
hand is attained at least twice, then $\inom_{\omega}(g(f))$
contains at least the terms $P_k(g)x_k^d$ corresponding to the
minimal coordinates $k$ and therefore is not a monomial. So for all
$g\in U$ we conclude that $T(g(I))=\Wc_n^{n-1}$. So $\gT(I)=\Wc_n^{n-1}$.
\end{proof}

For linear ideals $I\subset S_K$, that is, ideals generated by linear forms, the tropical variety of $I$ just depends on the matroid of $I$ as observed in \cite{ST02}. This matroid $M(I)$ on $N=\left\{1,\ldots,n\right\}$ is defined by declaring the circuits to be the minimal subsets $A$ of $N$ such that there exists a linear form in $I$ supported in variables with indices in $A$. Tropical varieties of matroids have been studied in \cite{ARKL}.

We explicitly compute the generic Gr\"obner fan and the generic tropical variety of linear ideals $I$. These just depend on the dimension of $S_K/I$ as fans.

Let $I\subset S_K$ be linear. Then a matrix $A=(a_{ij})\subset K^{t\times n}$ will be called a \emph{matrix of $I$}, if there exist the linear forms $f_i=\sum_{j=1}^n a_{ij} x_j$, such that $I=(f_1,\ldots,f_t)$. Note that choosing different generators of $I$ by taking linear combinations of the original ones corresponds to Gaussian operations on a given matrix of $I$. If $I\subset S_K$ is a linear ideal and $A$ is a matrix of $I$, then $\rank A=n-\dim (S_K/I)$.

Let $\dim(S_K/I)=m$ and $J\subset N=\left\{1,\ldots,n\right\}$ with $\left|J\right|=n-m$. Let $A$ be a matrix of $I$. If the minor of $A$ corresponding to the columns indexed by $J$ is nonzero, we can consider the reduced form $A_J$ of $A$ with respect to $J$. By this we mean the matrix obtained from $A$ by performing Gaussian elimination such that the square matrix of the columns corresponding to indices in $J$ is the identity matrix. For example, for $J=\left\{1,\ldots,n-m\right\}$ we have

$$A_{J}=\begin{pmatrix}
1 &  \cdots  & 0 & * & \cdots & *\\
\vdots & \ddots & \vdots & \vdots & \vdots & \vdots\\
0 & \cdots  & 1 & * & \cdots & *
\end{pmatrix},$$

where the $*$ represent any element of $K$.

For the generic situation note that if $A\subset K^{r\times n}$ is the matrix of $I$ and $g\in \GL_n(K)$, then we can consider
$g$ as a matrix $g\in K^{n\times n}$ and observe that the matrix product $Ag\subset K^{r\times n}$ is exactly the matrix of $g(I)$.
This is true, since for the generator $f_i$ of $I$ we have $$g(f_i)=\sum_j a_{ij}g(x_j)=\sum_{j}\sum_k a_{ij}g_{jk} x_k
=\sum_{k} \left(\sum_j a_{ij}g_{jk}\right) x_k ,$$ so the coefficient of $x_k$ in $g(f_i)$ is exactly the product of the
$i$-th row of $A$ and the $k$-th column of $g$.

\begin{lemma}\label{minornotzero}
Let $A\in K^{r\times n}$ of $\rank r$. Then there is a non-empty
Zariski-open subset $U\subset \GL_n(K)$ such that
\begin{enumerate}
\item every $r\times r$ minor of $Ag$ is non-zero for every $g\in U$,
\item every entry $*$ on the right hand side of $(Ag)_{J}$ as above is non-zero for $g\in U$ for every $J\subset N$ with $\left|J\right|=r$.
\end{enumerate}
\end{lemma}

\begin{proof}
The $r\times r$-minors of $Ag$ can be considered as polynomials in the $g_{ij}$. If one of these polynomials was the zero polynomial, that would mean, that the determinant of the corresponding submatrix is zero for all $g\in \GL_n(K)$, in particular for permutation matrices in $\GL_n(K)$ that swap columns of $A$. This implies that the determinant of all possible $r\times r$-submatrices of $A$ are zero and thus $\rank A<r$, which is a contradiction. So all $r\times r$-minors of $Ag$ are non-zero polynomials $\left\{f_1,\ldots,f_s\right\}$ in the $g_{ij}$. Thus we can choose $U$ as the set of all $g\in \GL_n(K)$ with $f_i(g)\neq 0$ for $i=1,\ldots,s$.

For the second statement we note that if every $r\times r$-minor of $Ag$ is non-zero, so is every $r\times r$-minor of $(Ag)_{J}$ for a fixed $J$, since Gaussian elimination preserves the rank of a matrix. So for $g\in U$ every $r\times r$-minor of $(Ag)_{J}$ is not zero. Now assume that some entry $*_{ij}$ for some $j\notin J$ of $(Ag)_{J}$ is equal to $0$. Consider the submatrix $B$ of $(Ag)_{J}$ consisting of the $r$ columns of $(Ag)_{J}$ corresponding to $J$, except that the $i$th column is replaced by the $j$th one. Then every entry in $i$th row of $B$ is zero, and thus $\det B=0$. But this is a contradiction to the fact that no $r\times r$-minor of $(Ag)_{J}$ is zero.
\end{proof}

The last statement together with \cite[Proposition 1.6]{ST} (or \cite[Proposition 1.4.4]{MATH}) shows that for a linear ideal $I$ with $\dim(S_K/I)=m$ generically the universal Gröbner basis consists of ${{n}\choose{m-1}}$ linear forms each supported on a different subset of size $n-m+1$ of $N$. Equivalently the matroid associated to $I$ is the uniform matroid of rank $n-m$ on $N$, see \cite[Example 9.13]{ST02}.

\begin{prop}\label{groebfanlinear}
Let $I\subset S_K$ be a linear ideal with $\dim (S_K/I)=m$.
\begin{enumerate}
\item
The generic Gr\"obner fan $\gGF(I)$ contains the following cones.
\begin{enumerate}
\item For $\omega\in \R^n$ with $\left\{i_1,\ldots,i_n\right\}=\left\{1,\ldots,n\right\}$ such that $$\omega_{i_1},\ldots,\omega_{i_{n-m}}<\omega_{i_{n-m+1}},\ldots,\omega_{i_n}$$ we have $$C[\omega]=\left\{\omega'\in\R^n: \omega'_{i_1},\ldots,\omega'_{i_{n-m}}<\omega'_{i_{n-m+1}},\ldots,\omega'_{i_n}\right\}.$$
\item For $\omega\in \R^n$ with $\left\{i_1,\ldots,i_n\right\}=\left\{1,\ldots,n\right\}$ such that $$\omega_{i_1},\ldots,\omega_{i_{n-m-t-1}}<\omega_{i_{n-m-t}}=\omega_{i_{n-m-t+1}}=\ldots=\omega_{i_{n-m+s}}<\omega_{i_{n-m+s+1}},\ldots,\omega_{i_n}$$ for $t\geq0,s\geq1$ we have that $C[\omega]$ is equal to the set $$\left\{\omega'\in\R^n:\omega'_{i_1},\ldots,\omega'_{i_{n-m-t-1}}<\omega'_{i_{n-m-t}}=\omega'_{i_{n-m-t+1}}=\ldots=\omega'_{i_{n-m+s}}<\omega'_{i_{n-m+s+1}},\ldots,\omega'_{i_n}\right\}.$$
\end{enumerate}
\item
The generic tropical variety $\gT(I)$ is equal to $\Wc_n^m$ as a fan.
\end{enumerate}
\end{prop}

\begin{proof}
Let $\omega\in \R^n$ such that after possibly renaming coordinates $\omega_1\leq\omega_2\leq\ldots\leq\omega_n$, and $\succ_{\omega}$ be a term order with $x_1\succ x_2\ldots\succ x_n$ which refines $\omega$. Let $A$ be a matrix of $I$ with $\rank A=r=n-m$. By \cite[Proposition 1.4.4]{MATH} the rows of the matrix $(Ag)_{J}$ for $J=\left\{1,\ldots,n-m\right\}$ are a reduced Gr\"obner basis of $g(I)$. For $g\in U$ as defined in Lemma \ref{minornotzero} the rows of $(Ag)_J$ correspond to linear forms $$l_i=x_i+\sum_{k=r+1}^n c_{ik} x_k$$ with $c_{ik}\neq 0$ for $i=1,\ldots, r$, $k=r+1,\ldots,n$. Now $\omega'\in \R^n$ is in the same Gr\"obner cone as $\omega$, if and only if $\inom_{\omega'} (l_i)=\inom_{\omega} (l_i)$ for $i=1,\ldots,r$. Since $\omega_1,\ldots,\omega_{n-m}\leq \omega_{n-m+1},\ldots,\omega_n$ this immediately implies $\omega'_1,\ldots,\omega'_{n-m}\leq \omega'_{n-m+1},\ldots,\omega'_n$. For every equality of some $\omega_i=\omega_k$ for $i\in\left\{1,\ldots,n-m\right\}$, $k\in \left\{n-m+1,\ldots,n\right\}$ the vector $\omega'$ has to fulfill the same equality such that $\inom_{\omega'}(l_i)=\inom_{\omega}(l_i)$. This completes the proof of the first part.

For the second statement we already know that $\gT(I)=\Wc_n^m$ as a set. On the other hand $\gT(I)$ is a subfan of the Gr\"obner fan $\gGF(I)$ as computed in Theorem \ref{groebfanlinear}. But $\Wc_n^m$ is a subfan of $\gGF(I)$, since the maximal cones of $\Wc_n^m$ are exactly the cones $$C=\left\{\omega\in\R^n:\omega_{i_1}=\ldots=\omega_{i_{n-m+1}}\leq\omega_{i_{n-m+2}},\ldots,\omega_{i_n}\right\}$$ of $\gGF(I)$. Hence $\gT(I)=\Wc_m^n$ as a fan.
\end{proof}

\begin{rem}
The second statement also follows from \cite{ARKL}, where Bergman fans of matroids are computed. In our case the matroid $M$ to consider is the uniform matroid of rank $n-m$ on $N$. The generic tropical variety of $\gT(I)$ is then the Bergman fan $\tilde{\Bc}(M)$ of \cite{ARKL} equipped with the coarse subdivision defined there.
\end{rem}

One implication of this is that the generic tropical variety of an ideal is generally not the $m$-skeleton of its generic Gr\"obner fan, since already for linear ideals $I$ the generic Gr\"obner fan $\gGF(I)$ has more $m$-dimensional cones than $\gT(I)$. In fact, for example the $m$-dimensional cone $C[\omega]$ with
$$
\omega_1<\omega_2=\dots=\omega_{n-m+2}<\omega_{n-m+3},\dots,\omega_n
$$
is an element of $\gGF(I)$, but not an element of $\gT(I)$.

\end{document}